\documentclass[a4paper,draft,reqno,12pt]{amsart}
\usepackage[english]{babel}
\usepackage{amsmath}
\usepackage{amssymb}
\usepackage{amscd}
\usepackage{amsthm}
\usepackage{euscript}
\newtheorem{theor}{Theorem}[section]
\newtheorem{prop}[theor]{Proposition}
\newtheorem{lem}[theor]{Lemma}

\newtheorem{pr}[theor]{Problem}
\newtheorem{cor}[theor]{Corollary}
\theoremstyle{definition}
\newtheorem{de}[theor]{Definition}
\newtheorem{const}[theor]{Construction}
\newtheorem{ex}[theor]{Example}
\theoremstyle{remark}
\newtheorem {re}[theor]{Remark}

\DeclareMathOperator{\Aut}{Aut}
\DeclareMathOperator{\SAut}{SAut}

\def\reg{{\rm reg}\,}

\def\GG{{\mathbb G}}

\def\KK{{\mathbb K}}

\def\ZZ{{\mathbb Z}}
\def\GA{{\mathbb G}_a}

\def\NN{{\mathbb N}}
\def\QQ{{\mathbb Q}}

\begin{document}
\date{}
\title[]{Generically flexible affine varieties with invariant divisors}
\author{Sergey Gaifullin}
\address{Lomonosov Moscow State University, Faculty of Mechanics and Mathematics, Department of Higher Algebra, Leninskie Gory 1, Moscow, 119991 Russia; \linebreak 
Moscow Center for Fundamental and Applied Mathematics, Moscow, Russia; \linebreak and \linebreak
HSE University, Faculty of Computer Science, Pokrovsky Boulvard 11, Moscow, 109028 Russia.}
\email{sgayf@yandex.ru}
\thanks{The article was prepared within the framework of the project “International academic cooperation” HSE University.}



\begin{abstract}
We construct examples of normal affine varieties $X$ of dimension $\geq 4$ such that the group of special automorphisms $\SAut(X)$ acts on $X$ with an open orbit $\mathcal{O}$ and the complement $X\setminus\mathcal{O}$ has codimension one.  
\end{abstract}

\maketitle

\section{Introduction}

Let $\KK$ be an algebraically closed field of characteristic zero and $\GA$ be its additive group. Assume $X$ is an irreducible affine algebraic variety over $\KK$. Let us denote by $\Aut(X)$ the group of regular automorphisms of $X$. 
Following \cite{AFKKZ} we consider the {\it subgroup of special automorphisms} $\SAut(X)$  of the group $\Aut(X)$ generated by all one-parameter unipotent subgroups, that is, subgroups in $\Aut(X)$ coming from all regular actions 
$\GA\times X\rightarrow X$. These algebraic one-parameter unipotent subgroups of $\Aut(X)$ we call {\it $\GA$-subgroups}. 

Recall that a regular point $x\in X$ is called {\it flexible} if the tangent space $\mathrm{T}_xX$ is spanned by tangent vectors to orbits of $\GA$-subgroups. We say that a variety $X$ is {\it flexible} if all regular points $x\in X$ are flexible. 

An action of a group $G$ on a set $Y$ is called {\it $m$-transitive} if for every two $m$-tuples of pairwise distinct points $(a_1, . . . , a_m)$ and $(b_1, . . . , b_m)$ on~$Y$ there exists an element $g\in G$ such that for all $i$ we have $g\cdot a_i =b_i$.
An action is {\it infinitely transitive} if it is $m$-transitive for all positive integers~$m$.

In \cite[Theorem~0.1]{AFKKZ} it is proved that for an irreducible affine variety $X$ of dimension $\geq 2$ the following conditions are equivalent:

(i) the group $\SAut(X)$  acts transitively on $X^{\mathrm{reg}}$;

(ii) the group $\SAut(X)$  acts infinitely transitively on $X^{\mathrm{reg}}$;

(iii) the variety $X$ is flexible.

There are a lot of examples of flexible varieties, see e.g.~\cite{AFKKZ, AKZ, Du, GSh, P}. 

In the same paper~\cite{AFKKZ} it is proved that each $\SAut(X)$-orbit is open in its closure. In particular, if $\SAut(X)$-orbit is dense in $X$, it is open. If $\SAut(X)$ acts on $X$ with an open orbit, the variety $X$
 is called {\it generically flexible}. This condition is equivalent to existence of at least one flexible point on $X$. Another equivalent condition is that any rational $\SAut(X)$-invariant function on $X$ is a constant. 
 
 Generic flexibility does not imply flexibility, but only one family of examples of this type is available in the literature. Namely,  a family of smooth Gizatullin surfaces $X$ which are unions of  an open $\SAut(X)$-orbit and a finite collection of $\SAut(X)$-fixed points
is given in~\cite{K}. So it is an interesting problem to build out new generically flexible, but not flexible varieties. In particular, we wonder whether the complement to an open $\SAut(X)$-orbit can contain a divisor. 

In this work we construct several examples of this type. It is relatively easy to obtain a non-normal affine flexible variety with a divisor in singular locus. In Section~\ref{tws}, we provide such an example in the category of toric varieties. At the same time any normal non-degenerate affine toric variety $X$ is flexible, so the group $\SAut(X)$ acts transitively on the regular locus $X^{\reg}$ and the complement to the open orbit has codimension at least two in $X$.

In Section~\ref{ms} for each $m\geq 4$ we construct an example of a normal affine $m$-fold $X$ such that the group $\SAut(X)$ acts on $X$ with an open orbit $\mathcal{O}$ and the complement $X\setminus\mathcal{O}$ has codimension one in $X$.  The variety $X$ is obtained as a categorical quotient of an affine trinomial hypersurface by a one-dimensional torus. In this case the complement  $X\setminus\mathcal{O}$ is a union of four prime divisors that are permuted by the group $\Aut(X)$. The intersection of these four divisors is an affine line $L$ that is the singular locus of $X$. Moreover, the group $\SAut(X)$ acts on $X^{\reg}$ with five orbits and $L$ consists of $\SAut(X)$-fixed points. At the same time, the group $\Aut(X)$ acts on $X$ with four orbits, see~Theorem~\ref{a}. 


The author is grateful to Ivan Arzhantsev and Kirill Shakhmatov for useful discussions.
\section{Preliminaries}
\subsection{Derivations}

Let $A=\KK[X]$ be the algebra of regular functions on an irreducible affine variety $X$.  

\begin{de}
A linear mapping $\delta\colon A\rightarrow A$ is called {\it derivation} if it satisfies the Leibniz rule: $\delta(ab)=a\delta(b)+b\delta(a)$.

A derivation is called {\it locally nilpotent} (LND) if for any $a\in A$ there is $n\in\NN$ such that $\delta^n(a)=0$.

A derivation is called {\it semisimple} if there exists a basis of $A$ consisting of $\delta$-semi-invariants. Recall that $a\in A$ is a $\delta$-semi-invariant if $\delta(a)=\lambda a$ for some $\lambda\in\KK$.
\end{de}
Having an LND $\delta$ one can consider its exponent 
$$
\exp(\delta)=\mathrm{id}+\delta+\frac{\delta^2}{2!}+\frac{\delta^3}{3!}+\ldots\in\Aut(A).
$$
Exponential mapping gives one-to-one correspondence between LNDs and $\GA$-subgroups in $\mathrm{Aut}(A)$:
$$
\delta\leftrightarrow \{\exp(s\delta)\mid s\in\KK\}.
$$
Having a $\GA$-action, one can obtain the corresponding LND as the derivative to this action at zero:
$\delta(f)=\frac{\mathrm{d} (s\cdot f)}{\mathrm{d} s}\mid_{s=0}$. Each LND  $\delta$ corresponds to a vector field, which we also denote by $\delta$. At each point~$p$ this vector field gives a tangent vector to the orbit of the $\GA$-subgroup corresponding to $\delta$. 

Similarly, if we have an action of the multiplicative group $(\KK^\times, \cdot)$ on~$A$, then its derivative at unity is a semisimple derivation, and this gives a bijection between $\KK^\times$-actions and semisimple derivations.

Let $F$ be an abelian group.

\begin{de}
An algebra $A$ is called {\it $F$-graded} if $$A=\bigoplus_{f\in F}A_f,$$ and $A_fA_g\subset A_{f+g}$ for all $f,g\in F$.
\end{de}

\begin{de}
A derivation $\delta\colon A\rightarrow A$ is called {\it $F$-homogeneous of degree $f_0\in F$} if for every $a\in A_f$ we have $\delta(a)\in A_{f+f_0}$.
\end{de}

Let $A$ be a finitely generated $\ZZ$-graded algebra. If is easy to see that any derivation $\delta$ can be decomposed onto a finite sum of homogeneous ones, i.e. 
$\delta=\sum_{i=l}^k\delta_i$, where $\delta_i$ is a homogeneous derivation of degree~$i$. 
Further when we write $\delta=\sum_{i=l}^k\delta_i$, we assume $\delta_l\neq 0$ and $\delta_k\neq 0$.

The following lemma follows from~\cite[Lemma~3.1]{FZ}, see also~\cite{Re} for part i). 
\begin{lem}\label{fl}
\,Let $A$ be a finitely generated $\ZZ$-graded algebra and $\delta\colon A\rightarrow A$ be a derivation. Suppose $\delta=\sum_{i=l}^k\delta_i$, where $\delta_i$ is a homogeneous derivation of degree $i$. Then

i) if $\delta$ is an LND, then $\delta_l$ and $\delta_k$ are LNDs;

ii) if $\delta$ is a semisimple derivation, then $l\neq 0$ implies $\delta_l$ is an LND (respectively $k\neq 0$ implies $\delta_k$ is an LND).
\end{lem}

If we have a $\mathbb{Z}^k$-grading on $A$, then applying the previous lemma several times we obtain the following statement.

\begin{lem}\label{ffll}
Let $A$ be a finitely generated $\ZZ^k$-graded algebra and $\delta\colon A\rightarrow A$ be a derivation. We have $\delta=\sum_{\gamma\in\mathbb{Z}^k}\delta_\gamma$, where $\delta_\gamma$ is a homogeneous derivation of degree $\gamma$. Suppose $\beta$ is a vertex of the convex hull of the set of degrees $\gamma$ such that $\delta_\gamma\neq 0$. Then

i) if $\delta$ is an LND, then $\delta_\beta$ is an LND;

ii) if $\delta$ is a semisimple derivation and $\beta\neq 0$, then $\delta_\beta$ is an LND.
\end{lem}

\begin{de}
The intersection of kernels of all LNDs on $X$ is called {\it the Makar-Limanov} invariant of $X$ and is denoted by $\mathrm{ML}(X)$. 
\end{de}
This invariant is a subalgebra of $\KK[X]$. This subalgebra coincides with the subalgebra of regular $\SAut(X)$-invariants, i.e. $\mathrm{ML}(X)=\KK[X]^{\SAut(X)}$. If $\mathrm{ML}(X)\neq \KK$, then $X$ is not generically flexible. The converse is not true, see~\cite[Section~4.2]{L} for an example of a variety $X$ with $\mathrm{ML}(X)=\KK$, which is not generically flexible. The field of rational $\SAut(X)$-invariants is called the {\it field Makar-Limanov invariant} and is denoted by $\mathrm{FML}(X)$. Generic flexibility is equivalent to $\mathrm{FML}(X)=\KK$. 

\subsection{Neutral component of the automorphism group}

Let us define the neutral component of $\Aut(X)$ following \cite{Ra}.

\begin{de}\label{dd} A family $\{\varphi_b, b\in B\}$ of automorphisms of a variety~$X$, where the parametrizing set $B$ is an irreducible algebraic variety, is an {\it algebraic family} if the map $B\times X\rightarrow X$ given by $(b,x)\mapsto \varphi_b(x)$ is a morphism.
\end{de}

\begin{de}\label{ddd} The neutral component $\Aut(X)^0$ of the group $\Aut(X)$ is the subgroup of
automorphisms that can be included in an algebraic family $\{\varphi_b, b\in B\}$ with an irreducible
variety as a base $B$ such that $\varphi_{b_0}=\mathrm{id}_X$ for some $b_0\in B$.
\end{de}

It is easy to check that $\Aut(X)^0$ is indeed a subgroup, see~\cite{Ra}.

If $G$ is a connected algebraic group and $G\times X\rightarrow X$ is a regular action, then we may take $B = G$ and consider the algebraic family $\{\varphi_g, g\in G\}$, where $\varphi_g(x)=g\cdot x$. So any automorphism defined by an element of $G$ is contained in $\Aut(X)^0$. In particular, every $\GG_a$- and every $\GG_m$-subgroup is contained in $\Aut(X)^0$. Here under $\GG_m$-subgroup we mean an algebraic subgroup of $\Aut(X)$ isomorphic to the multiplicative group of the ground field.

\subsection{Toric varieties}

In this section we give basic facts about toric varieties. More information one can find in ~\cite{CLSch,Ful}. An irreducible algebraic variety $X$ is called {\it toric}, if an algebraic torus $T=(\mathbb{K}^\times)^n$ acts algebraically on $X$ with an open orbit. We can assume the action of $T$ on $X$ to be effective. Note that we do not assume toric varieties to be normal. Let $X$ be affine. An affine toric variety $X$ corresponds to a finitely generated monoid $P$ of weights of $T$-semi-invariant regular functions. Let us identify the group of characters $\mathfrak{X}(T)$ with a free abelian group $M\cong\mathbb{Z}^n$. A vector $m\in\mathbb{Z}^n$ corresponds to the character~$\chi^m$. Since the open orbit on $X$ is isomorphic to $T$, we have an embedding of algebras of regular functions $\mathbb{K}[X]\hookrightarrow\mathbb{K}[T]$. Identifying the algebra $\mathbb{K}[X]$ with its image we obtain the following subalgebra graded by $P$:
$$
\mathbb{K}[X]=\bigoplus_{m\in P}\mathbb{K}\chi^m\subset \bigoplus_{m\in M}\mathbb{K}\chi^m=\mathbb{K}[T].
$$
Let us consider the vector space $M_{\mathbb{Q}}=M\otimes_{\mathbb{Z}}\mathbb{Q}$ over the field of rational numbers. The monoid $P$ generates the cone $\sigma^{\vee}=\mathbb{Q}_{\geq0}P\subset M_{\mathbb{Q}}$. Since~$P$ is finitely generated, the cone $\sigma^\vee$ is a finitely generated cone. Since the action of~$T$ on $X$ is effective, the cone $\sigma^\vee$ is not contained in any proper subspace of $M_\QQ$. 
The variety $X$ is normal if and only if the monoid $P$ is {\it saturated}, i.e. $P=M\cap \sigma^{\vee}$. If $P$ is not saturated, then we say that the monoid $P_{sat}=M\cap \sigma^{\vee}$ is the {\it saturation}  of the monoid~$P$. Elements of $P_{sat}\setminus P$ are called {\it holes} of $P$. Let us give some definitions according to \cite{TY}.

\begin{de}
An element $p$ of the monoid $P$ is called a {\it saturation point} of $P$, if the shifted cone $p+\sigma^{\vee}$ has no hole, i.e. $(p+\sigma^{\vee})\cap M\subset P$. 
 
A face $\tau$ of the cone $\sigma^{\vee}$ is called {\it almost  saturated}, if there is a saturation point of  $P$ in $\tau$. 
Otherwise $\tau$ is called a {\it nowhere saturated} face.
\end{de}
The maximal face, i.e. the whole cone $\sigma^\vee$, is always almost saturated.

We denote the lattice of one-parameter subgroups of the torus $T$ by~$N$. The lattice $N$ is the dual lattice to $M$. We denote the natural pairing
$M\times N\rightarrow \mathbb{Z}$ by $\langle \cdot,\cdot\rangle$. This pairing can be extended to a pairing between vector spaces $N_\QQ=N\otimes_\ZZ\QQ$ and $M_\QQ$. In the space~$N_\QQ$ we define the cone $\sigma$ dual to $\sigma^\vee$, by the rule
$$
\sigma=\{v\in N_\QQ\mid\forall w\in\sigma^\vee : \langle w,v\rangle\geq 0\}.
$$
The finitely generated polyhedral cone $\sigma$ is {\it pointed}, i.e. it does not contain any nonzero subspace.

A toric variety $X$ is called {\it degenerated} if $\KK[X]^\times\neq\KK^\times$. For a toric variety $X$ the following conditions are equivalent: 1) the variety $X$ is non-degenerate; 2)  the cone $\sigma^\vee$ is pointed; 3) the cone $\sigma$ is not contained in any hyperplane. Any toric variety is a direct product of a non-degenerate toric variety and a torus.

There is a bijection between $k$-dimensional faces of $\sigma$ and $(n-k)$-dimensional faces of~$\sigma^\vee$. A face $\tau\preccurlyeq\sigma$ corresponds to the face  $\widehat{\tau}=\tau^{\bot}\cap\sigma^\vee\preccurlyeq \sigma^\vee$. Also there is a bijection between $(n-k)$-dimensional faces of $\sigma^\vee$ and $k$-dimensional $T$-orbits on $X$. A face $\widehat{\tau}\preccurlyeq \sigma^\vee$ corresponds to the orbit, which is open in the set of zeros of the ideal 
$$I_{\widehat{\tau}}=\bigoplus_{m\in P\setminus\widehat{\tau}}\KK\chi^m.$$
The composition of these bijections gives a bijection between $k$-dimensional faces of the cone $\sigma$ and $k$-dimensional $T$-orbits. We denote the orbit corresponding to a face $\tau$  by~$O_\tau$.

\section{Non-normal examples}\label{tws}

Suppose $X$ is a non-normal flexible affine variety such that the singular locus $X^{\mathrm{sing}}$ has codimension one. Then $X$ is generically flexible and each prime divisor containing in $X^{\mathrm{sing}}$ is $\SAut(X)$-invariant. Indeed, $\SAut(X)$ permutes prime divisors containing in $X^{\mathrm{sing}}$. But $\SAut(X)$ is generated by $\GA$-subgroups. Each $\GA$-subgroup is connected, hence, it can not permute nontrivially a finite set of divisors. That is each prime divisor in $X^{\mathrm{sing}}$ is $\GA$-invariant for every $\GA$-subgroup. Hence, it is $\SAut(X)$-invariant. 

In this section we recall a necessary and sufficient condition for a non-normal toric variety to be flexible given in~\cite{BG}. Using this we can easily describe all flexible affine toric varieties that admit a $\SAut(X)$-invariant prime divisor. 

\begin{re}
It follows from results of~\cite{BG} that an affine toric variety can not be generically flexible but not flexible. In particular, all normal non-degenerate affine toric varieties are flexible, see also~\cite{AKZ}. 
\end{re}

Let us give a list of statements from~\cite{BG}. 

\begin{lem}[{Part of~\cite[Lemma~4]{BG}}]
Let $X$ be an affine toric variety, $\sigma$ be the corresponding cone, and $\rho$ be an extremal ray of~$\sigma$. Denote by $O_{\rho}$ the corresponding orbit. Then the following conditions are equivalent:

1) the face $\widehat{\rho}$ of the cone~$\sigma^{\vee}$ is almost saturated;

2) the orbit $O_{\rho}$ consists of smooth points.
\end{lem}

So, there exists a prime divisor $D\subseteq X^{\mathrm{sing}}$ if and only if there exists an extremal ray $\rho$ of $\sigma$ such that the face $\widehat{\rho}$ of the cone~$\sigma^{\vee}$ is nowhere saturated.

\begin{prop}{\cite[Corollary~2]{BG}}
An affine toric variety $X$ is flexible if and only if  there is no hyperplane in $N_\QQ$ containing all extremal rays~$\rho_i$ of the cone  $\sigma$ such that the face $\widehat{\rho_i}$ is almost saturated. 
\end{prop}

So, we obtain the following corollary. 
\begin{cor}
Let $X$ be an affine toric variety, and $\sigma$ be the corresponding cone. The variety $X$ is flexible with an $\SAut(X)$-invariant prime divisor if and only if the following conditions are satisfied:

1) there exists an extremal ray $\rho$ of $\sigma$ such that the face $\widehat{\rho}$ of the cone~$\sigma^{\vee}$ is nowhere saturated;

2)  there is no hyperplane in $N_\QQ$ containing all extremal rays $\rho_i$ of the cone  $\sigma$ such that the face $\widehat{\rho_i}$ is almost saturated.
\end{cor}

Let us give an explicit example. 
\begin{ex}
Let $P\subseteq \mathbb{Z}^3$ be the subset consisting of 
\begin{enumerate}
\item all integer points $(a,b,c)$ such that $a,b,c \geq 0$ and $a+b>c$;  
\item integer points $(a,b,a+b)$ such that $a,b \geq 0$ and $a$ is even. 
\end{enumerate}
Then $\sigma$ has four extremal rays $\rho_1=\QQ_{>0}(1,0,0)$, $\rho_2=\QQ_{>0}(0,1,0)$, $\rho_3=\QQ_{>0}(0,0,1)$, $\rho_4=\QQ_{>0}(1,1,-1)$. The faces $\widehat{\rho_1}$, $\widehat{\rho_2}$, and $\widehat{\rho_3}$ of $\sigma^\vee$ are almost saturated, and the face $\widehat{\rho_4}$ is nowhere saturated. Therefore, the corresponding affine toric variety $X$ is flexible with an $\SAut(X)$-invariant prime divisor. 

Let us denote $x=\chi^{(1,0,0)}$, $y=\chi^{(0,1,0)}$, $z=\chi^{(0,0,1)}$. It can be shown that 
\begin{multline*}
\KK[X]=\KK[x,y,x^2z, x^2z^2,yz]=\\=\KK[x,y,u,v,w]/(x^2v-u^2, x^2w-yu, uw-yv),
\end{multline*}
and the singular divisor is $\{x=y=u=0\}$.
\end{ex}

\section{Main results}\label{ms}

Let us consider the following hypersurface in the four-dimensional affine space:
$$
Y=\{xy^4+z^4+w^4=0\}.
$$ 
This is a trinomial hypersurface, which is a particular case of so-called trinomial varieties, see~\cite{H-W}. It is proved in~\cite[Theorem~1.2]{H-W} that such varieties are irreducible and normal. 
Let us fix an integer $m\geq 2$ and consider the direct product $Z=Y\times\mathbb{A}^m$ of $Y$ and an $m$-dimensional affine space. Denote by $u_1, u_2,\ldots, u_m$ coordinates on $\mathbb{A}^m$. Consider a $\mathbb{Z}$-grading $\mathfrak{G}$ on $\KK[Z]$ given by 
$$
\deg x=0,\qquad \deg y=\deg z=\deg w=1,\qquad \deg u_i=-1\text{ for all } i.
$$
This grading corresponds to the following $\KK^\times$-action: 
$$t\cdot (x,y,z,w,u_1,\ldots, u_m)=(x,ty,tz,tw,t^{-1}u_1, \ldots, t^{-1}u_m).$$
Let $X_m$ be the spectrum of the algebra of $\KK^\times$-invariants $\KK[Z]^{\KK^\times}$. It is known that $X_m$ is a normal irreducible variety. The algebra $\KK[X_m]$ is generated by 
$x$, $y_i=yu_i$, $z_i=zu_i$, and $w_i=wu_i$, where $1\leq i\leq m$.  Since generic $\KK^\times$-orbit on $Z$ has dimension one, we have $$\dim X_m=\dim Z-1=m+2.$$

\begin{lem}\label{mainl}
The subvariety $D=\{y_1=\ldots=y_m=0\}$ is a  union of four $\SAut(X_m)$-invariant prime divisors.
\end{lem}
\begin{proof}
Let us prove that the subset $D$ is $\SAut(X_m)$-invariant. This is equivalent to that the ideal $I=(y_1,\ldots, y_m)\subseteq \KK[X_m]$ is $\SAut(X_m)$-invariant.

Consider one more $\mathbb{Z}$-grading on $\KK[Z]$ given by 
$$
\deg{x}=-4, \qquad \deg y=1,\qquad \deg z=\deg w=\deg u_i=0 \text{ for all }i. 
$$
This grading induces the following $\ZZ$-grading  $\mathfrak{F}$ on $\KK[X_m]$:
$$
\deg x=\!-4, \deg y_i=1, \deg z_i=\deg w_i=0 \text{ for all }i.
$$
Suppose $\delta$ is an  $\mathfrak{F}$-homogeneous LND of $\KK[X_m]$ with negative degree. Then $\deg \delta(x)<\deg x=-4<0$. Therefore, $x\mid \delta(x)$. This implies $\delta(x)=0$, see~\cite[Principle~5]{Fr}. Let us consider the field $\mathbb{L}=\overline{\KK(x)}$, which is the algebraic closure of $\KK(x)$. Then $\delta$ induces a nonzero LND $\widehat{\delta}$ on the variety $X_m(\mathbb{L})$ over $\mathbb{L}$. But changing coordinates we can assume that $\mathbb{L}[X_m(\mathbb{L})]=\mathbb{L}[\widehat{Z}]^{\mathbb{L}^\times}$, where $\widehat{Z}$ is given in ($m+3$)-dimensional affine space with coordinates $\widehat{y},\widehat{z},\widehat{w},\widehat{u_1},\ldots,\widehat{u_m}$ by $\widehat{y}^4+\widehat{z}^4+\widehat{w}^4=0$ and 
$$t\cdot(\widehat{y},\widehat{z},\widehat{w},\widehat{u_1},\ldots,\widehat{u_m})=(t\widehat{y},t\widehat{z},t\widehat{w},t^{-1}\widehat{u_1},\ldots,t^{-1}\widehat{u_m}).$$   Therefore, the algebra $\mathbb{L}[X_m(\mathbb{L})]$ is generated by 
$$\widehat{y}_i=\widehat{y}\widehat{u_i},\qquad \widehat{z}_i=\widehat{z}\widehat{u_i},\qquad \widehat{w}_i=\widehat{w}\widehat{u_i}.$$
We have $\widehat{y}_i^4+\widehat{z}_i^4+\widehat{w}_i^4=0$. Since $\frac{1}{4}+\frac{1}{4}+\frac{1}{4}<1$, by the ABC-Theorem we have $\widehat{y}_i, \widehat{z}_i, \widehat{w}_i\in \mathrm{ML}(X(\mathbb{L}))\subseteq\mathrm{Ker}\,\widehat{\delta}$ , see~\cite[Theorem~2.48]{Fr}. This implies $\widehat{\delta}=0$. So, we obtain a contradiction. Hence, there is no LND~$\delta$ on $\KK[X_m]$ with negative  $\mathfrak{F}$-degree. 

Now let $\delta$ be an arbitrary LND on $X_m$. Consider the decomposition $\delta=\sum\limits_{i=l}^k\delta_i$ onto  a sum of $\mathfrak{F}$-homogeneous derivations.  By Lemma~\ref{fl}, the derivation $\delta_l$ is locally nilpotent. Therefore, $l\geq 0$. Hence, for each summand $\delta_j$ we have $j\geq 0$. Then $\deg\delta_j(y_i)=1+j>0$. Therefore, $\delta_j(y_i)\in I$. This implies $\delta(y_i)\in I$. Thus, $I$ is $\delta$-invariant. Since $\delta$ is an arbitrary LND, $I$ is $\SAut(X)$-invariant.

Let $\varepsilon_1,\varepsilon_2, \varepsilon_3, \varepsilon_4$ be all 4-roots of $-1$. Let us prove that $D$ is the union of four subsets: $D_1$, $D_2$, $D_3$, and $D_4$, where
$$
D_j=\{y_i=z_i+\varepsilon_jw_i=0\text{ for all }i\}.
$$
Indeed, let $p$ be a point on $D$. Since the quotient mapping $\pi\colon Z\rightarrow X_m$ is surjective, we can take $q\in Z$ such that $\pi(q)=p$. If for all $i$ we have $u_i(q)=0$, then $z_i(p)=w_i(p)=0$ and $p$ belongs to all $D_j$. Suppose there exists $1\leq s\leq m$ such that $u_s(q)\neq 0$. Since $y_s(p)=0$, we obtain $y(q)u_s(q)=0$, hence $y(q)=0$. This implies $z^4(q)+w^4(q)=0$. Hence, for some $\varepsilon_j$ we have $z(q)+\varepsilon_jw(q)=0$. So, $z_i(p)+\varepsilon_jw_i(p)=0$ for all~$i$. Therefore, $p\in D_j$. Note that $\KK[D_j]$ is a polynomial algebra in variables $x,z_1,\ldots, z_m$. So, $D_j$ is a prime divisor isomorphic to the affine space~$\mathbb{A}^{m+1}$. 

Since $D=D_1\cup D_2\cup D_3\cup D_4$ is $\SAut(X_m)$-invariant, the group $\SAut(X_m)$ permutes~$D_j$. Therefore, each~$D_j$ is $\SAut(X_m)$-invariant. 

\end{proof}

We need the following elementary lemma. 

\begin{lem}\label{el}
Suppose a point $p$ of an affine variety $X$ belongs to a closed $\SAut(X)$-invariant subvariety $S\subseteq X$ of dimension $d$. Then the linear span of tangent vectors to $\GA$-orbits at $p$ has dimension $\leq d$.
\end{lem}
\begin{proof}
Let us prove this assertion by induction on $d$. 

If $d=0$, then each $\GA$-orbit of $p$ is trivial and tangent vectors to $\GA$-orbits span the zero-dimensional space.

Suppose for $d<d_0$ the assertion is proved. Let us prove it for $d=d_0$. If $p$ is a regular point of $S$, then $\dim \mathrm{T}_pS=d$. Since all tangent vectors to $\GA$-orbits at $p$ belong to $\dim \mathrm{T}_pS$, the linear span of these vectors has dimension $\leq d$. Now, let $p$ be a singular point of $S$. Then $p\in S^{\mathrm{sing}}$, where $\dim S^{\mathrm{sing}}\leq d-1$. Each $\GA$-action on $X$ induces a $\GA$-action on $S$, which preserves $S^{\mathrm{sing}}$. Therefore, $S^{\mathrm{sing}}$ is a  $\SAut(X)$-invariant subvariety of $X$. By inductive hypothesis, the linear span of tangent vectors to $\GA$-orbits at $p$ has dimension $\leq d-1$.
\end{proof}

An elementary consequence of the previous lemma is the following statement. 
\begin{lem}\label{esl}
Let $X$ be a normal affine variety of dimension $n$. If the linear span of tangent vectors to $\GA$-orbits at a point $p\in X$ has dimension $\geq n-1$, then $p$ is a regular point. 
\end{lem}
\begin{proof}
Since $X$ is normal, the codimension of $X^{\mathrm{sing}}$ is al least 2. So, $X^{\mathrm{sing}}$ is a closed $\SAut(X)$-invariant subvariety of dimension $\leq n-2$. 
\end{proof}

Let us investigate $\Aut(X_m)$-orbits on $X_m$.  First of all we need to write down explicitly some LNDs on $X_m$.   

\begin{const}\label{vc}
Let us consider the following LNDs $\xi_z, \xi_w$, and $\zeta_{ij}$, $1\leq i\neq j\leq m$ on $Z$:
$$
\xi_z(x,y,z,w,u_1,\ldots,u_m)=(4z^3u_1^3,0,y^4u_1^3,0,0,\ldots,0);
$$
$$
\xi_w(x,y,z,w,u_1,\ldots,u_m)=(4w^3u_1^3,0,0,y^4u_1^3,0,\ldots,0).
$$
The LND $\zeta_{ij}$ is defined by $\zeta_{ij}(u_i)=u_j$ and $\zeta_{ij}$ equals zero on all other variables. 

All these LNDs are homogeneous of degree zero with respect to the $\ZZ$-grading $\mathfrak{G}$. Hence, they define the following LNDs of~$\KK[X]$:
$$
\delta_z(x)=4z_1^3, \qquad \delta_z(z_1)=y_1^4, \qquad   \delta_z(z_i)=  y_1^3y_i  \text{ for }i>1;   
$$
$$
\delta_w(x)=4w_1^3, \qquad \delta_w(w_1)=y_1^4, \qquad   \delta_w(w_i)=  y_1^3y_i  \text{ for }i>1;   
$$
$$
\rho_{ij}(y_i)=y_j,\qquad \rho_{ij}(z_i)=z_j, \qquad \rho_{ij}(w_i)=w_j.
$$
\end{const}

\begin{lem}\label{ooo}
The set $\mathcal{O}=X_m\setminus D$ is one $\SAut(X_m)$-orbit.
\end{lem}
\begin{proof}
It is sufficient to show that each point $p$ in~$\mathcal{O}$ is flexible. Indeed, the $\SAut(X_m)$-orbit of a flexible point is open and since $X_m$ is irreducible, all flexible points belong to the unique open orbit.

Firstly let us prove that if $y_i(p)\neq 0$ for all $1\leq i\leq m$, then $p$ is a flexible point. 
 Let us use LNDs $\delta_z,\delta_w, \rho_{12}, \rho_{21}, \rho_{31},\ldots ,\rho_{m1}$ from Construction~\ref{vc}. 
 Suppose 
 $$\mu_z\delta_z(p)+\mu_w\delta_w(p)+\lambda_1 \rho_{12}(p)+\lambda_2 \rho_{21}(p)+\ldots+\lambda_m\rho_{m1}(p)=0.$$
 Since
 $
 \rho_{ij}(y_i)=y_j 
 $
 and all other considered LNDs take $y_i$ to zero, we have $\lambda_i=0$. But $\delta_z(z_1)=y_1^4$, $\delta_w(z_1)=0$, $\delta_w(w_1)=y_1^4$, $\delta_z(w_1)=0$. So, $\mu_1=\mu_2=0$. Thus, we have proved that all the considered LNDs are linearly independent at $p$. Since $\dim X_m=m+2$, by Lemma~\ref{esl}, $p$ is regular. We conclude that $p$ is flexible.  So, all points with  $y_1(p)\neq 0,\ldots, y_m(p)\neq 0$ are flexible. 
 
Now let us take an arbitrary point $p\notin D$. There exists $j$ such that $y_j(p)\neq 0$. Suppose  $y_i(p)=0$. Then there exists $s\in \KK\setminus\{0\}$ such that for $q=\exp(s\rho_{ij})(p)$ we have $y_i(q)\neq 0$. Note that for any $k\neq i$ we have $y_k(q)=y_k(p)$.  Applying such automorphisms several times we obtain a point $r\in X_m\setminus D$ such that $y_i(r)\neq 0$ for all $i$, i.e. $r$ is flexible. Since we take $p$ to $r$ by automorphism, the point $p$ is also flexible.  
 \end{proof}

Let $L$ be the line 
$$\{y_1=\ldots=y_m=z_1=\ldots=z_m=w_1=\ldots=w_m=0\}\subset X_m.$$ 

\begin{lem}\label{sipo}
All points in $L$ are singular in $X$. Moreover, $\dim \!\mathrm{T}_pX_m=3m+1$ for every $p\in L$. 
\end{lem}
\begin{proof}
 Fix a point $p=(x,0,\ldots,0)\in L$. Let us prove that there exist $3m+1$ curves $\gamma_1(t),\ldots, \gamma_{3m+1}(t)\subseteq X_m$ containing $p$  with linearly independent tangent vectors at $p$. Since $\dim T_{(x,0,0,0)}Y=4$, there exist curves $\eta_1(t),\ldots, \eta_4(t)$ such that their tangent vectors at $(x,0,0,0)$ are linearly independent. Let us denote $\frac{\mathrm{d} \eta_i}{\mathrm{d} t}|_{t=0}(x,0,0,0)=(a_i,b_i,c_i,d_i)$. 
 We can assume that vectors $(b_1,c_1,d_1), (b_2,c_2,d_2), (b_3,c_3,d_3)$ are linearly independent. 
 

For $1\leq j\leq m$, $1\leq i\leq 3$ let us put $u_j(t)=1$, $u_k(t)=0$ for $k\neq j$, $(x,y,z,w)=\eta(t)$. We obtain a curve on~$Z$. The image of this curve on $X_m$ we denote $\gamma_{3j+i-3}(t)$. The tangent vector to this curve at $p=\gamma_{3j+i-3}(0)$ is 
$$
(a_i,0,\ldots,0,b_i,0,\ldots,0,c_i,0,\ldots,0,d_i,0,\ldots,0), 
$$
where nonzero coordinates have numbers $1$, $j+1$, $m+j+1$ and $2m+j+1$. 

Also we define $\gamma_{3m+1}$ as the image of the curve given by $u_k(t)=1$ for all $k$, $(x,y,z,w)=\eta_4(t)$ in $Z$. Its tangent vector at $p=\gamma_{3m+1}(0)$ is 
$$
(a_4,b_4,\ldots,b_4,c_4,\ldots,c_4,d_4,\ldots,d_4).
$$ 

It is easy to see that all $3m+1$ tangent vectors to $\gamma_k$ at $p$ are linearly independent. But since $\KK[X_m]$ has $3m+1$ generators, the tangent space~$\mathrm{T}_pX_m$ can not have dimension greater than $3m+1$. 
\end{proof}

Analogically to Lemma~\ref{ooo} we prove the following lemma. 

\begin{lem}\label{ood}
The set $U_j=D_j\setminus L$ is an $\SAut(X_m)$-orbit. 
\end{lem}
\begin{proof}
It is sufficient to show that for each point $p\in U_j$ the tangent space $\mathrm{T}_pU_j=\mathrm{T}_pD_j$ is generated by the tangent vectors to orbits of $\GA$-actions. Indeed, if the orbit $\SAut(X_m)p$ has dimension $\leq m$, then the closure $\overline{\SAut(X_m)p}$ is a closed $\SAut(X_m)$-invariant subvariety of dimension $\leq m$. By Lemma~\ref{el}, the tangent vectors can not generate $(m+1)$-dimensional space. Hence, orbits of all points $p\in U_j$ are open in $D_j$. Since $D_j$ is irreducible, $U_j$ is contained in one $\SAut(X_m)$-orbit~$O$. By Lemma~\ref{esl}, this orbit consists of regular points of $X_m$. Lemma~\ref{sipo} provides $O=U_j$. 

So, we need to show that for each $p\in U_j$ there are $m+1$ LNDs such that corresponding vector fields are linearly independent at $p$. Let us take the following LNDs from Construction~\ref{vc}: $\delta_z, \rho_{12},\rho_{21},\ldots, \rho_{m1}$. We use that same notations for the corresponding vector fields. Suppose $p\in D$ and $z_i(p)\neq 0$ for all~$i$.  Assume 
 $$\mu\delta_z(p)+\lambda_1 \rho_{12}(p)+\lambda_2 \rho_{21}(p)+\ldots+\lambda_m\rho_{m1}(p)=0.$$
We have
$
\delta_z(x)(p)=4z_1^3(p)\neq 0
$
and $\rho_{ij}(x)(p)=0$. Therefore, $\mu=0$. Also we have $\rho_{ij}(z_i)(p)=z_j(p)\neq0$ and $\rho_{kj}(z_i)=0$ for $k\neq i$. This implies $\lambda_i=0$. 
So, these $m+1$ vector fields are linearly independent at~$p$. 

Now let us take any $p\in U_j$. Since $p\notin L$, there exists $k$ such that $z_k(p)\neq 0$. Note that $\rho_{ik}(z_i)=z_k$. Hence, $\rho_{ik}(z_i)(p)\neq 0$. Therefore there exists $s\in\KK$ such that $\exp (s\delta_3)(p)=q$, where $z_i(q)\neq 0$. Moreover $z_l(q)=z_l(p)$ for each $l\neq i$. So, there exists an automorphism taking $p$ to $r$, where $z_j(r)\neq 0$ for all $j$. Since tangent vectors to $\GA$-orbits span a $(m+1)$-dimensional space at $r$, so they do at $p$ either. 
\end{proof}

\begin{lem}\label{rere}
The set $D=D_1\cup D_2\cup D_3\cup D_4$ from Lemma~\ref{mainl} is an $\Aut(X_m)$-orbit. 
\end{lem}  
\begin{proof}
Multiplying all $w_i$ by $\varepsilon$, where $\varepsilon^4=1$, we can permute transitively four prime divisors $D_j$. 
\end{proof}

Now we can describe orbits of $\Aut(X_m)$ and $\SAut(X_m)$. 

\begin{theor}\label{a} 

i) There are five $\SAut(X_m)$-orbits on $X_m^{\mathrm{reg}}$: the open one 
$$\mathcal{O}=X_m\setminus (D_1\cup D_2\cup D_3 \cup D_4)$$ 
and four orbits 
$$U_j=D_j\setminus\{z_u=z_v=w_u=w_v=0\},\qquad 1\leq j\leq 4.$$

ii) Singular points $X_m^{\mathrm{sing}}$ form a line of $\SAut(X_m)$-fixed points. This line coincides with $L$.

iii) There are two $\Aut(X_m)$-orbits on $X^{\mathrm{reg}}$: the open one $\mathcal{O}$ and a reducible subvariety
$U_1\cup U_2\cup U_3\cup U_4$ of codimension one.

iv) There are two $\Aut(X_m)$-orbits on $X_m^{\mathrm{sing}}$: a punctured line $\{(x,0,\ldots,0), x\neq 0\}$ and a point $(0,\ldots,0)$.  
\end{theor}
\begin{proof}
The assertion i) follows from Lemma~\ref{mainl}, Lemma~\ref{ooo}, Lemma~\ref{sipo}, and Lemma~\ref{ood}.

To prove~ii), we need to prove that each point $p=(x,0,\ldots,0)$ on~$X_m$ is $\SAut(X_m)$-fixed. Let us denote by $I$ the ideal 
$$(y_1,\ldots, y_m, z_1,\ldots, z_m, w_1,\ldots,w_m)\subseteq\KK[X_m].$$  
It is sufficient to prove that for any LND~$\delta$ we have $\delta(x)\in I$. Then  $\delta(x)(p)=0$. 
Let us consider $\mathbb{Z}^2$-grading on $\KK[X_m]$ defined by 
$$\deg x=(0,-4), \deg y_i=(1,1), \deg z_i=\deg w_i=(1,0)\text{ for all }i.$$
Suppose $\zeta$ is a $\mathbb{Z}^2$-homogeneous LND on $\KK[X_m]$ of degree $(\alpha,\beta)$. 
In the  proof of Lemma~\ref{mainl} we have proved that there is no nonzero LND~$\zeta$ with $\zeta(x)=0$. Therefore, $\beta\geq 4$, otherwise $x\mid \zeta(x)$. We have 
$\deg(\zeta(z_i))=\deg(\zeta(w_i))=(l,k),$
where $k\geq 4$. Therefore, each monomial in $\zeta(z_i)$ and $\zeta(w_i)$ is divisible by a product $\prod y_i^{s_i}$, where $\sum s_i\geq 4$. If $\zeta(z_1)=\zeta(w_1)=0$, then $\zeta(x)=0$, and hence, $\zeta=0$. This implies $\alpha\geq 3$.  

Now let $\delta$ be an LND on $X_m$. We can decompose this LND onto a sum of $\mathbb{Z}^2$-homogeneous derivations: $\delta=\sum\limits_{\gamma\in\mathbb{Z}^2}\delta_\gamma$. Derivations corresponding to vertices of the convex hull of degrees of summands are LNDs. Therefore, for each $\gamma=(c,d)$, we have $c\geq 3$.   

So, we have
$
\delta(x)=\sum\limits_{\gamma\in\mathbb{Z}^2}\delta_\gamma(x), 
$
where $\deg\delta_\gamma(x)=(c,d-4)$.
Since $c>0$, we have $\delta_\gamma(x)\in I$. So, $\delta(x)\in I$. 

By Lemma~\ref{rere}, $\Aut(X_m)$ permutes $D_j$ transitively. So, applying~$i)$ we obtain that all regular points are covered by no more than two $\Aut(X_m)$-orbits: $\mathcal{O}$ and $U_1\cup U_2\cup U_3\cup U_4$.  To prove iii), we need only to show that these two sets are not contained in the same $\Aut(X_m)$-orbit.  Each point from $\mathcal{O}$ is flexible. Since $D_j$ is $\SAut(X_m)$-invariant, each point on $D_j$ is not flexible. Since $\Aut(X_m)$ can not take a flexible point to a nonflexible one, we obtain the goal.

Now let us prove iv). We can consider the action of a one-dimensional torus $\KK^\times$ corresponding to the $\mathbb{Z}$-grading $\mathfrak{F}$, or equivalently the induced $\mathbb{Z}$-grading for $\mathbb{Z}=\{(0,n)\}\subseteq\mathbb{Z}^2$. 
That is $t\cdot x=t^{-4}x$, $t\cdot y_i=ty_i$, $t\cdot z_i=z_i$, and $t\cdot w_i=w_i$. This action restricted to $X_m^{\mathrm{sing}}=L$ has two orbits: a punctured line $\{(x,0,\ldots,0), x\neq 0\}$ and a point $(0,\ldots,0)$. So we need to prove that $(0,\ldots,0)$ is an $\Aut(X_m)$-fixed point. To do this it is sufficient to prove that $(0,\ldots,0)$ is $\KK^\times$-fixed for any $\KK^\times$-action on $X_m$. Each nontrivial $\KK^\times$-action corresponds to a nonzero semisimple derivation~$\delta$ on $\KK[X]$. Let us consider the decomposition of $\delta$ onto a sum of $\mathbb{Z}^2$-homogeneous derivations $\delta=\sum\limits_{\gamma\in\mathbb{Z}^2}\delta_\gamma$.  By Lemma~\ref{ffll}(ii), derivations corresponding to nonzero vertices of the convex hull $P$ of degrees of summands are LNDs. As it is proved above, all nonzero vertices of $P$ have the form $(a,b)$, where $a\geq 3$. Hence, for every $\gamma=(c,d)\neq (0,0)$ we have $c>0$. Therefore, $\delta_\gamma(x)\in I$. Let us denote by $J$ the ideal in $\KK[X]$ generated by $x,y_1,\ldots,y_m,z_1,\ldots,z_m, w_1,\ldots,w_m$. By definition, $I\subseteq J$. So, $\delta_\gamma(x)\in J$ for all nonzero $\gamma$.
We have $\deg \delta_{(0,0)}(x)=\deg x=(0,-4)$. So, $x\mid \delta_{(0,0)}(x)$. Therefore, $\delta_{(0,0)}(x)\in J$. Thus, $\delta(x)\in J$. 

Analogically we can prove that $\delta^m(x)\in J$ for all positive integers $m$. Indeed, 
$$\delta^m=\left(\sum\limits_{\gamma\in\mathbb{Z}^2}\delta_\gamma\right)^m=\delta_{(0,0)}^m+(\delta_{\gamma_1}\circ\ldots\circ \delta_{\gamma_m}+\ldots). $$
Here $\deg \delta_{(0,0)}^m(x)=\deg x=(-4,0)$, hence $\delta_{(0,0)}^m(x)\in J$. And for any $(\gamma_1,\ldots, \gamma_m)\neq (0,\ldots,0)$, we have 
$$\deg \delta_{\gamma_1}\circ\ldots\circ \delta_{\gamma_m}(x)=(r,s), \text{ where }r>0.$$
Therefore, $\delta_{\gamma_1}\circ\ldots\circ \delta_{\gamma_m}(x)\in J$. This implies $\delta^m(x)\in J$. Hence, for the corresponding $\KK^\times$-action, we have $t\cdot x\in J$. 
Therefore, the point $(0,\ldots,0)$ is $\KK^\times$-fixed. 
\end{proof}

\begin{cor}
The variety $X_m$ is generically flexible, but not flexible. 
\end{cor}

Using Theorem~\ref{a} it is easy to describe orbits of the neutral component $\Aut(X_m)^0$. 
\begin{cor}
The neutral component $\Aut(X_m)^0$ has seven orbits on $X_m$. Five orbits consist of regular points: the open one $\mathcal{O}$ and four orbits 
$U_j,\ 1\leq j\leq 4$, and two orbits consist of singular points: a punctured line $\{(x,0,0,0,0,0,0), x\neq 0\}$ and a point $(0,0,0,0,0,0,0)$. 
\end{cor}
\begin{proof}
Since $\SAut(X_m)\subseteq \Aut(X_m)^0$, by Theorem~\ref{a} i) $\mathcal{O}$ is contained in one $\Aut(X_m)^0$-orbit. But by Theorem~\ref{a} iii)  $\mathcal{O}$ is an $\Aut(X_m)$-orbit. So,  $\mathcal{O}$ is a $\Aut(X_m)^0$-orbit. Similarly each $U_j$ is contained in one $\Aut(X_m)^0$-orbit. Since $U_j$ are irreducible components of $X_m^{\reg}\setminus\mathcal{O}$, each element of $\Aut(X_m)^0$ permutes $U_j$. Connectedness of  $\Aut(X_m)^0$ provides that each $U_j$ is  $\Aut(X_m)^0$-invariant. So, each $U_j$ is an  $\Aut(X_m)^0$-orbit.

By Theorem~\ref{a}~iv), the point $(0,0,0,0,0,0,0)$ is $\Aut(X_m)$-fixed. So, it is $\Aut(X_m)^0$-fixed. All other singular points of $X_m$ lie in one $\KK^\times$-orbit, and hence in one $\Aut(X_m)^0$-orbit. 
\end{proof}

\begin{re}
One can build one more family of generically flexible varieties admitting $\SAut(X)$-invariant divisors. For this purpose one can start not from $Y=\{xy^4=z^4+w^4\}$, but from $Y'=\{xy^n=z_1^n+\ldots+z_k^n\}$, where $k+1<n$. Then, as above, $Z'=Y'\times \mathbb{A}^m$. The desired variety~$X'_m$ is obtained as the categorical quotient $Z'/\!/\mathbb{K}^\times$. 
\end{re}

\section{Concluding remarks and questions}\label{lll}

Let us say that $X$ is {\it flexible in codimension one} if $X$ admits an open $\SAut(X)$-orbit $\mathcal{O}$ and the codimension of $X\setminus \mathcal{O}$ in $X$ is at least~$2$. Results of Section~\ref{ms} show that generic flexibility does not imply flexibility in codimension one. But for some types of varieties this implication is still true.

Let us start with the following assertion.

\begin{prop}\label{tr}
If an irreducible affine variety $X$ is generically flexible and $D\subseteq X$ is an $\SAut(X)$-invariant prime divisor, then $D$ is not principle. 
\end{prop}
\begin{proof}
Since $X$ is generically flexible, we have $\KK[X]^\times=\KK^\times$. Suppose $D=\mathrm{div}(f)$. Since $D$ is  $\SAut(X)$-invariant, we have~$f$ is $\SAut(X)$-semi-invariant, i.e. $\SAut(X)$ multiplies $f$ by a character. But $\SAut(X)$ is generated by $\GA$-subrgoups. Since $\GA$-subgroups are unipotent, they have no nontrivial characters. Therefore, $\SAut(X)$ has no nontrivial character. That is $f$ is $\SAut(X)$-invariant. This contradicts to generic flexibility of~$X$.
\end{proof}

\begin{cor}
If $\KK[X]$ is a UFD and $X$ is generically flexible, then~$X$ is flexible in codimension one.
\end{cor}
\begin{proof}
If $\KK[X]$ is a UFD, then each prime divisor on $X$ is principle. Then the assertion of the corollary follows from Proposition~\ref{tr}.
\end{proof}

\begin{pr}
Find new classes of varieties for which generic flexibility implies flexibility in codimension one. 
\end{pr}

One more type of divisors that can not be $\SAut(X)$-stable is as follows. Let us recall that a $\KK^\times$-action is {\it hyperbolic} if it has two non-zero weight spaces for weights of different signs. Suppose $D$ is the set of stable points of non-hyperbolic $\mathbb{K}^\times$-action. This is equivalent to that $\KK[X]=\bigoplus\limits_{i\geq 0}\KK[X]_i$. And $D=\mathbb{V}(I)$, where $I=\bigoplus\limits_{i>0}\KK[X]_i$. The following proposition is a particular case of~\cite[Proposition~3]{GSh}.
\begin{prop}\label{hml}
Suppose a prime divisor $D$ containing a regular point of $X$ is the set of stable points of non-hyperbolic $\mathbb{K}^\times$-action. Then $D$ is not $\SAut(X)$-invariant. 
\end{prop}
Moreover, if $X$ is normal, under conditions of the previous proposition one can write down explicitly an LND $\delta$ such that $D$ is not $\exp(\delta)$-invariant, see~\cite[Lemma~5.8]{BGSh}. For this purpose one should fix a rational function $h$ such that $\mathrm{div}(h)=-D+E$, where $E$ is an effective divisor. Then for $f\in\KK[X]_j$ one put $\delta(f)=jfh$.  

Let us state the following problems.

\begin{pr}
Find new sufficient conditions on a prime divisor $D$ in which it can not be $\SAut(X)$-invariant. 
\end{pr}

An important class of varieties are closures of orbits of affine algebraic group actions. Such varieties have at least the acting group $G$ in $\Aut(X)$. Often this allows to prove generic flexibility and even flexibility of these varieties. Recall that an irreducible variety is {\it almost homogeneous} if it admits an action of an algebraic group with an open orbit. We have no example of a non-flexible normal affine almost homogeneous $G$-variety with only constant invertible regular functions for a reductive group $G$. Note that if we omit the condition of reducteveness for $G$, then we have some examples. 
\begin{ex}
Let $X=\{xy^2=z^2-1\}$ be a Danielewski surface. We can define an action of the semidirect product $G=\mathbb{K}^\times\rightthreetimes \mathbb{K}$ on $X$ be the rule:
$$
(t,s)\cdot (x,y,z)=(t^2x+2t^2sz+t^2s^2y^2, t^{-1}y, z+sy^2).
$$
It is easy to see that this action has an open orbit $\mathcal{O}=\{y\neq 0\}$. The variety $X$ is normal and admits only constant invertible regular functions, see for example \cite[Theorem~1.2]{H-W}.
But by \cite{ML}, the Makar-Limanov invariant $\mathrm{ML}(X)=\KK[y]\neq\KK$. Therefore, $X$ is not generically flexible. 
\end{ex}

In~\cite[Section~5]{AFKKZ2} the following problem is stated. 
\begin{pr}\label{qgr}
Characterize flexible varieties among the normal almost homogeneous affine varieties. 
\end{pr}

Note that for toric varieties~\cite{AKZ} and horospherical varieties~\cite{GSh} the only condition for a normal variety to be flexible is that invertible functions are constants. Also the same criterion is proved for smooth varieties, see~\cite[Theorem~2]{GSh} and~\cite[Theorem~5.6]{AFKKZ} for the case of semisimple group.

Let us state the following necessary and sufficient condition for a variety with a locally transitive action of an algebraic group to be flexible in codimension one. 
\begin{prop}
Suppose $X$ is a normal irreducible affine variety with only constant invertible regular functions. Assume a connected reductive group~$G$ acts on $X$ with an open orbit. The variety $X$ is not flexible in codimension one if and only if there is $\SAut(X)$-invariant prime divisor in the complement to the open $G$-orbit.  
\end{prop}
\begin{proof}
By definition, if $X$ is flexible in codimension one, then it admits an open orbit $\mathcal{O}$ such that codimension of $X\setminus \mathcal{O}$ is at least $2$. Therefore, there is no $\SAut(X)$-invariant prime divisor on $X$.  

Conversely, let us assume that there is no $\SAut(X)$-invariant prime divisor in the complement to the open $G$-orbit. We define by $H$ the group generated by $\SAut(X)$ and the maximal torus of the acting group~$G$. Then by \cite[Proposition~4]{GSh} if $X$ is not generically flexible, then there exists an $H$-invariant prime divisor. A connected reductive group $G$ is generated by its maximal torus and root $\GA$-subgroups. Therefore, the image of the acting group $G$ in $\Aut(X)$ is contained in~$H$. Hence, an $H$-invariant prime divisor is an $\SAut(X)$-invariant divisor in the complement to the open $G$-orbit. A contradiction. So, $X$ is generically flexible. Since $G$ acts on $X$ by automorphisms, the open orbit consists of flexible points. If each prime divisor in the complement to the open $G$-orbit contains a flexible point, it contains an open  set of flexible points, and hence, $X$ is flexible in codimension one. Suppose $D\subseteq X\setminus \mathcal{O}$ is a prime divisor. Since $D$ is not $\SAut(X)$-invariant, there exists a $\GA$-subgroup that takes a point $p\in D$ to a point $q\notin D$. We can assume that $p$ does not belong to any other irreducible component of $X\setminus \mathcal{O}$. Then the $\GA$-orbit of $p$ do not contained in any irreducible component of $X\setminus \mathcal{O}$. Since this orbit is irreducible, we can conclude that this $\GA$-orbit contains a point from $\mathcal{O}$, i.e. a flexible point. Therefore,~$p$ is flexible.  
\end{proof}
It was proved in~\cite{AZ} that there is no $\SAut(X)$-invariant divisor in the complement to the open orbit in case of affine spherical varieties.  So, we obtain the following corollary. 
\begin{cor}
Normal affine spherical varieties with only constant invertible regular functions are flexible in codimension one.
\end{cor}

Spherical varieties is an important class of varieties. So let us state one more problem, which also have been stated in~\cite[Section~5]{AFKKZ2}. 

\begin{pr}
Are all normal affine spherical varieties with only constant invertible regular functions flexible?
\end{pr}

In \cite{AFKKZ} the following proposition is proved. (In the case when $X$ is an affine space, this also follows from Gromov-Winkelmann theorem). 
\begin{prop}{\cite[Corollary~4.18]{AFKKZ}}
Let $X$ be an affine variety. Suppose that the group $\SAut(X)$ acts with an open orbit $\mathcal{O}\subseteq X$. Then for any finite subset $Z\subseteq \mathcal{O}$ and for any closed subset $Y\subseteq X$ of codimension $\geq 2$ with $Z\cap Y=\varnothing$ there is an orbit $C\cong \mathbb{A}^1$ of a $\GA$-action on~$X$ which does not meet $Y$ and passes through each point of $Z$.
\end{prop}

Note that if $D$ is an $\SAut(X)$-invariant prime divisor of a generically flexible variety, then we can take $Y=D$.

\end{document}